\documentclass[12pt]{amsart}
\usepackage{tikz-cd}
\usepackage{hyperref}
\usepackage[shortlabels]{enumitem}
\usepackage{xcolor}
\usepackage{amsmath,amssymb,amsthm,amscd}
\newtheorem{theorem}{Theorem}
\newtheorem{lemma}[theorem]{Lemma}

\newtheorem{proposition}[theorem]{Proposition}
\theoremstyle{definition}
\newtheorem{remark}{Remark}
\usepackage[margin=3cm]{geometry}
\newcommand{\Z}{\mathbb{Z}}

\newcommand{\F}{\mathbb{F}}
\newcommand{\Q}{\mathbb{Q}}

\newcommand{\Gal}{\textrm{Gal}}

\newcommand{\im}{\textrm{im}\thinspace}

\newcommand{\oo}{\mathcal{O}}

\newcommand{\ssm}{\smallsetminus}

\newcommand{\End}{\textrm{End}}

\newcommand{\GL}{\textrm{GL}}

\newcommand{\tors}{\textrm{tors}}

\newcommand{\cQ}{\mathcal{Q}}

\newcommand{\Ell}{\mathcal{L}}

\newcommand{\Gss}{G^{\textrm{ss}}}

\newcommand{\oQ}{\overline{\Q}}
\newcommand{\Cyc}{\textrm{Cyc}}

\usepackage{mathrsfs}
\usepackage{url}
\date{}
\author{Tyler Genao}
\thanks{Email: \texttt{genao.5@osu.edu}.}
\title{New Isogenies of Elliptic Curves Over Number Fields}
\begin{document}
\begin{abstract}
Using Galois representations, we analyze fields of definition of cyclic isogenies on elliptic curves to prove the following uniformity result: for any number field $F$ which has no rational CM, under GRH there exists an effectively computable constant $B:=B(F)\in\Z^+$ such that for any finite extension $L/F$ whose degree $[L:F]$ is coprime to $B$, one has for all elliptic curves $E_{/F}$ that any $L$-rational isogeny on $E$ is $F$-rational. For any number field $F$, under GRH we also prove results for the mod-$\ell$ Galois representations of non-CM elliptic curves with an $F$-rational isogeny of uniformly large prime degree $\ell$.
\end{abstract}
\maketitle

\section{Introduction}
Given a number field $F$, it is of great interest to understand isogenies of elliptic curves which are rational over $F$, in analogy to understanding rational torsion points. When $F=\Q$, a classic theorem of Mazur \cite{Maz78} shows that the only possible prime degrees of $\Q$-rational isogenies are $\ell\in\lbrace 2,3,5,7,11,13,17,19, 37,43,67,163\rbrace$. Following Mazur's work, Kenku \cite{Ken82} has shown that all possible degrees of $\Q$-rational isogenies with cyclic kernel are the integers $n\in[1,19]\cup \lbrace 21,25,27,37,43,67,163\rbrace$. 
For general number fields $F$, if $F$ has no rational \textit{complex multiplication} (henceforth CM), then there exists a uniform, effectively computable bound on degrees of $F$-rational isogenies under the \textit{generalized Riemann hypothesis} (henceforth GRH); see \cite{LV14} and \cite{Ban23}. However, it is currently unknown whether a uniform bound can hold unconditionally.

Nonetheless, one can prove uniformity results on isogenies of elliptic curves defined over number fields \textit{of a particular degree}. For example, if $L$ is a number field of odd degree, then for any elliptic curve $E_{/L}$ that is $\oQ$-isogenous to a $\Q$-rational elliptic curve, all $L$-rational isogenies on $E$ of prime degree $\ell$ must have $\ell\leq 37$; see \cite[Theorem 1.1]{CN21}. This type of result can be leveraged to study torsion from the family of $\Q$-curves, see also \cite{BN}.

In this paper, we prove uniformity results for cyclic isogenies of elliptic curves base-changed from $F$. We study such isogenies through analyzing the curve's Galois representations. Recall that for an elliptic curve $E$ defined over a number field $F$, for each integer $n\in\Z^+$ there exists an action of $G_F:=\Gal(\oQ/F)$ on the $n$-torsion subgroup $E[n]$ of $E$. The information of this Galois action is encoded in the \textit{mod-$n$ Galois representation of $E$.} Once we fix a basis of $E[n]$, we can write this representation as $\rho_{E,n}\colon G_F\rightarrow \GL_2(\Z/n\Z)$. This action extends to an action on the set of cyclic order $n$ subgroups of $E$. If $\phi$ is an isogeny from $E$ whose kernel $C$ is cyclic, then the the field of definition of $\phi$, written as $F(C)$, is the fixed field of the stabilizer of $C$. One key observation towards understanding the rationality of $\phi$ is that the degree $[F(C):F]$ is equal to the size of the orbit of $C$ under the action of $G_F$, and so $\phi$ is $F$-rational if and only if the orbit of $C$ is trivial.

Studying rationality of isogenies through this action of Galois can be useful. As an example, combinative work of several mathematicians, culminating in recent work of Furio and Lombardo \cite{FL}, shows that if $E_{/\Q}$ is a non-CM elliptic curve, then for any prime $\ell>37$, the mod-$\ell$ Galois representation $\rho_{E,\ell}\colon G_\Q\rightarrow \GL_2(\F_\ell)$ is either surjective,
or its image equals the normalizer of a nonsplit Cartan group mod-$\ell$. This fact, combined with a classification of possible images for $\rho_{E,\ell}(G_\Q)$ when $\ell\leq 37$ (see e.g. \cite{GJN20}), implies that the orbit of any cyclic order $n$ subgroup of $E$ is either trivial, or has a size that is a multiple of $2,3,5,7,11,13,17$ or $37$. Consequently, for any number field $L$ whose degree over $\Q$ is coprime to $2\cdot 3\cdot 5\cdot 7\cdot 11\cdot 13\cdot 17\cdot 37=18888870$, for any non-CM elliptic curve $E_{/\Q}$, an isogeny $\phi$ from $E$ is $L$-rational if and only if $\phi$ is $\Q$-rational. In particular, when $\phi$ is cyclic, the degree of $\phi$ is bounded uniformly by Mazur and Kenku's results.

With the above example in mind, let us introduce some terminology. Given an elliptic curve $E_{/F}$ and a cyclic $n$-isogeny with kernel $C\subseteq E(\oQ)$, we say that $C$ (and the isogeny) is \textbf{new} if $C$ is not $F$-rational.
For an extension $L/F$, we say that $C$ is \textbf{new for $L/F$} if $C$ is $L$-rational but not $F$-rational; equivalently, $C$ is new for $L/F$ if $F(C)\subseteq L$ and $[F(C):F]>1$. For example, the preceding paragraph shows that if $L$ is a number field with $[L:\Q]$ coprime to $18888870$, then for any non-CM elliptic curve $E_{/\Q}$, there are no new isogenies on $E$ for $L/\Q$.

The main theorem of this paper is a uniformity result for new isogenies on elliptic curves for $L/F$, where $F$ is any number field with no rational CM.
This theorem can be viewed as a generalization of \cite[Theorem 1]{Gen24}, which concerns new torsion points on elliptic curves for $L/F$.
\begin{theorem}\label{Thm_IsogeniesUnderBaseChange}
Let $F$ be a number field with no rational CM. Then under GRH, there exists an effectively computable constant $B:=B(F)\in\Z^+$ such that for any finite extension $L/F$ whose degree $[L:F]$ is coprime to $B$, there are no new isogenies for $L/F$ on any elliptic curve over $F$.
\end{theorem}

As will be made precise in the following section, the degree of the field of definition of an $\ell$-isogeny on $E_{/F}$ is determined by the size of its kernel's orbit under the action of $G_{F}$. When $F$ has no rational CM and $\ell$ is sufficiently large with respect to $F$, under GRH the group $E[\ell]$ of $\ell$-torsion points on $E$ is an irreducible $\F_\ell$-module, which restricts the possible images $\rho_{E,\ell}(G_{F})$ of the mod-$\ell$ Galois representation of $E$ over $F$. This implies that the image $\rho_{E,\ell}(G_{F})$ is ``relatively uniformly large" in $\GL_2(\F_\ell)$, as described in Theorem \ref{Thm_BaseChangeImprovement} below. 
This theorem can be viewed as progress towards Serre's uniformity question for number fields larger than $\Q$.
\begin{theorem}\label{Thm_BaseChangeImprovement}
Let $F$ be a number field with no rational CM. Then under GRH, there exists an effectively computable constant $c:=c(F)\in\Z^+$ such that for all primes $\ell>c$, one has the following: for an elliptic curve $E_{/F}$, if the image $G:=\rho_{E,\ell}(G_{F})$ does not equal $\emph{GL}_2(\F_\ell)$, then it is contained in the normalizer of a Cartan group in $\emph{GL}_2(\F_\ell)$ up to conjugacy.

Let $N_s(\ell)$ and $N_{ns}(\ell)$ denote the normalizer of the split and nonsplit Cartan group mod-$\ell$, respectively, and let $G(\ell)$ be as defined in $\S \ref{Sect_SubgpsOfGL2}$. Then up to a choice of basis for $E[\ell]$, one of the following holds:
\begin{enumerate}[a.]
\item If $G\subseteq N_s(\ell)$ then $C_s(\ell)^{e}\subseteq G$ for some $e\in \lbrace 1,2,3,4,6\rbrace$, and $G$ contains the subgroup of scalar matrices. Furthermore, the index of $G$ in $N_s(\ell)$ satisfies
\[
[N_s(\ell):G]\mid \gcd(\ell-1,e).
\]
\item If $G\subseteq N_{ns}(\ell)$ then $G=N_{ns}(\ell)$ or $G(\ell)$. If $\ell\equiv 1\pmod 3$ then $G=N_{ns}(\ell)$.
\end{enumerate}
\end{theorem}
The conclusion of Theorem \ref{Thm_BaseChangeImprovement} is an improvement of \cite[Theorem 2]{Gen24} in the case where $G$ is contained in the normalizer $N_{ns}(\ell)$ of the nonsplit Cartan group mod-$\ell$, up to conjugacy. More specifically, when $\ell$ is large, if up to conjugacy $G$ is also contained in the nonsplit Cartan group $C_{ns}(\ell)$, then it follows that $G$ is diagonalizable over the quadratic extension $\F_{\ell^2}$ of $\F_\ell$. Thus, $E[\ell]\otimes _{\F_\ell}\overline{\F_\ell}$ has a degree 1 associated character with values in $\F_{\ell^2}$. By Theorem \ref{Thm_LV}, this implies that $F$ has rational CM, or that GRH fails for $F[\sqrt{-\ell}]$. Therefore, assuming both GRH and that $F$ has no rational CM, we may conclude that $G$ cannot be contained in $C_{ns}(\ell)$ up to conjugacy when $\ell$ is large; thus, Theorem \ref{Thm_BaseChangeImprovement} follows at once from \cite[Theorem 2]{Gen24}.
 
Theorem \ref{Thm_IsogeniesUnderBaseChange} is in part a consequence of the following result, which can be viewed as a generalization of \cite[Proposition 3.3]{CN21} from $\Q$ to $F$.
\begin{proposition}\label{Prop_EvenDegreeFieldOfDefn}
Let $F$ be a number field with no rational CM. Then under GRH, there exists an effectively computable constant $c:=c(F)\in\Z^+$ such that for all primes $\ell> c$
and for all elliptic curves $E_{/F}$, the field of definition for an $\ell$-isogeny of $E$ has even degree over $F$.
\end{proposition}

In the final section of this paper, we will prove a uniformity result for elliptic curves defined over a number field $F$, where we do not assume that $F$ has no rational CM. In this theorem, part a. is complementary to Theorem \ref{Thm_BaseChangeImprovement}, providing a result towards Serre's uniformity question over $F$; part b. is a consequence of part a., and provides a necessary condition for the existence of new isogenies of large prime degree. 
\begin{theorem}\label{Thm_ModEllRep_ofEllCurve_With1_EllIsogeny}
Let $F$ be a number field. Assume GRH, and fix a prime $\ell\geq \max\lbrace c(F), 72[F:\Q]\rbrace$ with $c(F)$ as in Proposition \ref{Prop_EvenDegreeFieldOfDefn}. If $E_{/F}$ is an elliptic curve with exactly one $F$-rational $\ell$-isogeny, then $E$ must be non-CM. Furthermore, one has the following:
\begin{enumerate}[a.]
\item Up to conjugacy, the image $G:=\rho_{E,\ell}(G_{F})$ is contained in $B(\ell)$, the Borel subgroup of $\emph{GL}_2(\F_\ell)$. If $G\subseteq B(\ell)$, then one has 
\[
[B(\ell):G]\mid 864[F:\Q].
\]
\item For an order $\ell$ subgroup $C\subseteq E(\oQ)$, either $[F(C):F]=1$ or $[F(C):F]=\ell$.
\end{enumerate}
\end{theorem}
\begin{remark}\label{Rmk_ImplicitConstant}
The theorems and proposition above use the same constant $c(F)$: this is the least positive integer such that for all primes $\ell\geq c(F)$, one has $\ell\geq \max\lbrace 74, 15[F:\Q]+2\rbrace$, $\ell$ is unramified in $F$ and $\ell\not\in S_{F}$, where $S_F$ is defined in Theorem \ref{Thm_LV}.
\end{remark}

\section*{Acknowledgments}
The author thanks the referee for their comments and suggestions, which helped improve the quality of this paper.
\section{Galois Action on Torsion and Isogenies}
In this section, we will set up some notation and terminology for the rest of this paper, as well as describe and prove some basic results.
\subsection{The action of Galois on $n$-torsion points of elliptic curves}
For an integer $n>1$, we let $\GL_2(n):=\GL_2(\Z/n\Z)$ be the group of $2\times 2$ invertible matrices over $\Z/n\Z$. If we let $e_1:=\begin{bmatrix}
1\\0
\end{bmatrix}$ and $e_2:=\begin{bmatrix}
0\\1
\end{bmatrix}$ be the unit column vectors over $\Z/n\Z$, then $\GL_2(n)$ acts by left multiplication on the $\Z/n\Z$-module $V$ generated by $e_1$ and $e_2$, which is free of rank two. 

Once and for all, fix an algebraic closure $\oQ$ of $\Q$.
Given an elliptic curve $E_{/F}$ and an integer $n\in\Z^+$, the $n$-torsion subgroup $E[n]:=E(\oQ)[n]$ is a free rank two $\Z/n\Z$-module \cite[Corollary 6.4]{Sil09}; a choice of basis $\lbrace P,Q\rbrace$ for $E[n]$ gives a $\Z/n\Z$-module isomorphism $E[n]\cong V$.
In fact, the absolute Galois group $G_{F}:=\Gal(\oQ/F)$ acts $\Z/n\Z$-linearly on $E[n]$. With respect to a basis $\lbrace P,Q\rbrace$ of $E[n]$, this action is explicitly described by the \textit{mod-$n$ Galois representation}
\[
\rho_{E,n,P,Q}\colon G_{F}\rightarrow \GL_2(n).
\]
This map is as follows: if $\sigma\in G_{F}$ acts on $E[n]$ via $P^\sigma=aP+cQ$ and $Q^\sigma=bP+dQ$, then one has $\rho_{E,n,P,Q}(\sigma)=\begin{bmatrix}
a&b\\
c&d
\end{bmatrix}$. By the Weil pairing, the image of the determinant of $\rho_{E,n,P,Q}(G_F)$ is equal to the image of the mod-$n$ cyclotomic character on $G_F$.

 We will usually write $G:=\rho_{E,n}(G_F)$ instead of $\rho_{E,n,P,Q}(G_F)$ for the image of the mod-$n$ Galois representation over $F$; this notation is well-defined up to conjugacy, but we will often work with an implicitly fixed basis. One always has $G\cong \Gal(F(E[n])/F)$, where $F(E[n])$ is the $n$-division field of $E$ over $F$.

\subsection{Cyclic isogenies of elliptic curves}
Given an algebraic extension $F/\Q$ and two elliptic curves ${E_1}_{/F}$ and ${E_2}_{/F}$, an \textit{$F$-isogeny from ${E_1}$ to $E_2$} is a nontrivial $F$-rational morphism of curves $\phi\colon E_1\rightarrow E_2$ which fixes the point at infinity: $\phi(O_{E_1})=O_{E_2}$. Such an isogeny induces a group homomorphism $\phi\colon E_1(F)\rightarrow E_2(F)$ \cite[Theorem 4.8]{Sil09}. Being $F$-isogenous defines an equivalence relation on elliptic curves over $F$, and the equivalence class of an elliptic curve is called its \textit{$F$-isogeny class.} 

Given an isogeny $\phi\colon E_1\rightarrow E_2$, the kernel $C:=\ker\phi$ is a finite subgroup of $E_1[\tors]:=E_1(\oQ)[\tors]$; if $C$ is cyclic of order $n$, we say that $\phi$ is a \textit{cyclic $n$-isogeny.} Since $\phi$ is $F$-rational, it follows that $C$ is stable under the action of $G_F$: that is, for all $\sigma\in G_F$ one has $C^\sigma=C$. Conversely, any cyclic order $n$ subgroup $C\subseteq E_1[\tors]$ which is stable under the action of $G_F$ induces an $F$-rational isogeny on $E_1$ whose kernel is $C$ \cite[Exercise 3.13.e]{Sil09}. For this reason, we sometimes refer to cyclic order $n$ subgroups of $E_1[n]$ as cyclic $n$-isogenies; when $n=\ell$ is prime, an $\ell$-isogeny is automatically cyclic.

For an elliptic curve $E_{/F}$, the action of $G_F$ on $E[n]$ extends to an action on its set of cyclic order $n$ subgroups. In general, if we let $\Cyc(V)$ denote the set of nontrivial cyclic $\Z/n\Z$-submodules of the free rank two $\Z/n\Z$-module $V:=\langle e_1,e_2\rangle$, then the action of $\GL_2(n)$ on $V$ extends to an action on $\Cyc(V)$: for $C=\langle v\rangle \in \Cyc(V)$ and $m\in \GL_2(n)$, we set $m\cdot C:= \langle m\cdot v\rangle$. We will denote the set of cyclic order $n$ subgroups in $E$ by $\Cyc(E,n):=\Cyc(E[n])$.
Matters simplify in the case where $n=\ell$ is prime: one has $\# \Cyc(V)=\ell+1$, since the nontrivial order $\ell$ subgroups of $V$ are $\langle e_1\rangle$ and $C_k:=\langle ke_1+e_2\rangle$, where $0\leq k<\ell$.

To prove Proposition \ref{Prop_EvenDegreeFieldOfDefn}, we will need the following small but useful lemma.
\begin{lemma}\label{Lemma_DistinctCkSubgroups}
Fix an integer $n\in\Z^+$, and let $V$ be the $\Z/n\Z$-span of unit column vectors $e_1,e_2$ over $\Z/n\Z$. Then for all integers $1\leq k,m<n$, one has that $\langle ke_1+e_2\rangle=\langle me_1+e_2\rangle$ if and only if $k=m$. 
\end{lemma}
\begin{proof}
If $\langle ke_1+e_2\rangle =\langle me_1+e_2\rangle$, then $me_1+e_2\in \langle ke_1+e_2\rangle$, and so for some $x\in\Z$ we have 
\[
me_1+e_2=x(ke_1+e_2)=xke_1+xe_2.
\]
By $\Z/n\Z$-linear independence of $e_1$ and $e_2$, it follows that $m\equiv xk\pmod n$ and $1\equiv x\pmod n$; we deduce that $m\equiv k\pmod n$, whence from $1\leq k,m<n$ we conclude that $k=m$.
\end{proof}

\subsection{Orbits and fields of definition}\label{Subsect_FieldsOfDefn}
Given an elliptic curve $E_{/F}$ and a torsion point $R\in E[\tors]$, the \textit{field of definition of $R$ over $F$,} denoted by $F(R)$, is the extension of $F$ of least degree over which $R$ is rational. Equivalently, it is the fixed field of the stabilizer of $R$ under the action of $G_F$ on $E[n]$. 

Say $R\in E(\oQ)$ has order $n$, and let $G:=\rho_{E,n}(G_F)$. We will write $\oo_G(R)$ to denote the orbit of $R$ under $G$ (also called the \textit{$G$-orbit} of $R$). A key observation regarding orbits and fields of definition is the following. Since the stabilizer of $R$ under $G$ is $\Gal(F(E[n])/F(R))$, the orbit-stabilizer theorem implies that
\[
\# \oo_G(R)=[F(R):F].
\]
This size is invariant under a change of basis for $E[n]$. 

A similar description applies to the field of definition of a cyclic isogeny. Given an elliptic curve $E_{/F}$ and a cyclic $n$-isogeny $C\in \Cyc(E,n)$, the field of definition of $C$ over $F$, denoted by $F(C)$, is the least extension of $F$ over which $C$ is stable under the action of Galois. By Galois theory, $F(C)$ is also the fixed field of the stabilizer of $C$ under the action of $G_F$ on $\Cyc(E,n)$. The orbit-stabilizer theorem then implies that
\[
\# \oo_G(C)=[F(C):F],
\]
which is invariant under a change of basis for $E[n]$.

\subsection{Subgroups of $\GL_2(\ell)$}\label{Sect_SubgpsOfGL2}
In this subsection, we will define the groups which appear in Theorem \ref{Thm_BaseChangeImprovement}.

For $\ell\geq 3$, let us define the split Cartan group mod-$\ell$ as the subgroup of diagonal matrices,
\[
C_s(\ell):=\left\lbrace \begin{bmatrix}
a&0\\
0&d\end{bmatrix}\in \GL_2(\ell)\right\rbrace.
\]
Its normalizer is denoted as $N_s(\ell)$; explicitly, we have
\[
N_s(\ell)=C_s(\ell)\cup \begin{bmatrix}
0&1\\
1&0
\end{bmatrix}C_s(\ell).
\]

We will also define the nonsplit Cartan group mod-$\ell$, denoted $C_{ns}(\ell)$, as follows. For a prime $\ell\in\Z^+$, we will write $\F_\ell:=\Z/\ell\Z$.
Fix the least positive integer $\epsilon$ which generates $\F_\ell^\times$: then we have 
\[
C_{ns}(\ell)=\left\lbrace \begin{bmatrix}
a&b\epsilon\\
b&a
\end{bmatrix}\in \GL_2(\ell)\right\rbrace.
\]
This is the regular representation of $\F_\ell[\sqrt{\epsilon}]^\times$ acting on itself via multiplication with respect to the basis $\lbrace 1,\sqrt{\epsilon}\rbrace$. Its normalizer is
\[
N_{ns}(\ell)=C_{ns}(\ell)\cup \begin{bmatrix}
1&0\\
0&-1
\end{bmatrix}C_{ns}(\ell).
\]
Let us also define the Borel subgroup mod-$\ell$ as the subgroup of upper triangular matrices,
\[
B(\ell):=\left\lbrace \begin{bmatrix}
a&b\\
0&d
\end{bmatrix}\in \GL_2(\ell)\right\rbrace.
\]
Also appearing in Theorem \ref{Thm_BaseChangeImprovement} is the mod-$\ell$ subgroup
\[
G(\ell):=\left\langle \begin{bmatrix}
1&0\\
0&-1
\end{bmatrix},C_{ns}(\ell)^3\right\rangle,
\]
which has index $3$ in $N_{ns}(\ell)$ when $\ell\geq 5$. Finally, we will use $Z(\ell)$ to denote the subgroup of scalar matrices in $\GL_2(\ell)$.

\subsection{Orbits under a subgroup}
In this subsection, we will describe how orbits under the action of a finite group compare to orbits under its subgroups. This is particularly important for Theorem \ref{Thm_ModEllRep_ofEllCurve_With1_EllIsogeny} (and the part of the proof of Theorem \ref{Thm_BaseChangeImprovement} from \cite{Gen24}). Given a finite group $G$ which acts on a set $X$, one has for any subgroup $H\subseteq G$ that
\begin{equation}\label{Eqn_SubgroupOrbitDivisibility}
\#\oo_G(x)\mid [G:H]\cdot \#\oo_H(x).
\end{equation}

As noted in $\S \ref{Subsect_FieldsOfDefn}$, for an elliptic curve $E_{/F}$ and an integer $n\in\Z^+$, for any cyclic subgroup $C\in \Cyc(E,n)$ one has $\# \oo_{\rho_{E,n}(G_F)}(C)=[F(C):F]$. Thus, for any subgroup $N\subseteq \GL_2(n)$ with $\rho_{E,n}(G_F)\subseteq N$, we have by Equation \eqref{Eqn_SubgroupOrbitDivisibility} that
\[
\#\oo_N(C)\mid [N:\rho_{E,n}(G_F)]\cdot [F(C):F].
\]
In particular, if a prime $p\mid \#\oo_N(C)$ does not divide $[N:\rho_{E,n}(G_F)]$, then $p\mid [F(C):F]$. 

\subsection{The isogeny character}
In this subsection, we will describe some results on the \textit{isogeny character} of an elliptic curve. This is necessary for the proofs of Theorems \ref{Thm_IsogeniesUnderBaseChange} and \ref{Thm_ModEllRep_ofEllCurve_With1_EllIsogeny}, and Proposition \ref{Prop_EvenDegreeFieldOfDefn}. 

Recall that for an elliptic curve $E_{/F}$ with a finite cyclic subgroup $C\subseteq E(\oQ)$ of order $n$, if $C$ is $G_{F}$-stable, then the action of $G_{F}$ on $C$ induces a homomorphism $r\colon G_{F}\rightarrow (\Z/n\Z)^\times$ called the \textit{isogeny character of $C$.} This character is described as follows: writing $C:=\langle P\rangle$, for $\sigma\in G_{F}$ one has 
\[
P^\sigma=r(\sigma)\cdot P.
\] 

In Theorem \ref{Thm_ModEllRep_ofEllCurve_With1_EllIsogeny}, 
part b. will quickly follow from part a. The proof of part a. analyzes the isogeny character of an elliptic curve over $F$, using the following uniformity result from Larson and Vaintrob. Recall that a CM elliptic curve $E_{/F}$ has \textit{rationally defined CM over $F$} if its CM field is contained in $F$; in such a case, we say that $F$ has \textit{rational CM.}
\begin{theorem}\cite[Theorem 1]{LV14}\label{Thm_LV}
Let $F$ be a number field. Then there exists a finite set $S_{F}$ of prime numbers such that the following holds. Let $\ell\in \Z^+$ be a prime not in $S_{F}$, and let $E_{/F}$ be an elliptic curve. Suppose that $E[\ell]\otimes_{\F_\ell}\overline{\F_\ell}$ is reducible, with a degree $1$ associated character $r\colon G_{F}\rightarrow \overline{\F_\ell}^\times$. Then one of the following holds.
\begin{enumerate}[(1)]
\item There exists a CM elliptic curve $E'_{/F}$ with rationally defined CM over $F$, and associated characters $\psi,\psi'\colon G_{F}\rightarrow \overline{\F_\ell}^\times$ such that
\[
\rho_{E',\ell}(G_{F})\otimes_{\F_\ell}\overline{\F_\ell}=\emph{im}\begin{bmatrix}
\psi&0\\
0&\psi'
\end{bmatrix}
\]
up to conjugation, and
\[
r^{12}=\psi^{12}.
\]
\item GRH fails for $F[\sqrt{-\ell}]$, and one has
\[
r^{12}=\chi_\ell^{6}
\]
where $\chi_\ell\colon G_{F}\rightarrow \F_\ell^\times$ is the mod-$\ell$ cyclotomic character. Moreover, in this case one has that $E[\ell]$ is reducible over $\F_\ell$, and $\ell\equiv 3\pmod 4$.
\end{enumerate}
\end{theorem}
\begin{remark}
An effectively computable bound on the product of the primes from $S_{F}$ is given in \cite[Theorem 7.9]{LV14}. For certain quadratic fields, a significantly better bound is given by Banwait \cite[Corollary 1.9]{Ban23}. By Remark \ref{Rmk_ImplicitConstant}, improved bounds on primes from $S_F$ will sharpen the constant $c(F)$ which appears in Theorems \ref{Thm_BaseChangeImprovement} and \ref{Thm_ModEllRep_ofEllCurve_With1_EllIsogeny},
and Proposition \ref{Prop_EvenDegreeFieldOfDefn}. 
\end{remark}
\begin{remark}\label{Rmk_AssocChar}
For a prime $\ell\not\in S_{F}$, if an elliptic curve $E_{/F}$ has an $F$-rational $\ell$-isogeny, then $E[\ell]$ is reducible and the conclusions of Theorem \ref{Thm_LV} hold for the isogeny character $r\colon G_{F}\rightarrow \F_\ell^\times$. 
In part (1) of Theorem \ref{Thm_LV}, if we assume that $\ell$ is unramified in the endomorphism ring of $E'$, we find that the CM mod-$\ell$ Galois representation $G':=\rho_{E',\ell}(G_{F})$ is contained in the split or nonsplit Cartan subgroup of $\GL_2(\ell)$, up to conjugacy. In the split case, $E'[\ell]$ is a reducible $G_{F}$-module, and the associated character $\psi$ takes values in $\F_\ell^\times$; in particular, $\psi$ is the isogeny character of an $F$-rational $\ell$-isogeny on $E'$. 
However, in the nonsplit case, $E'[\ell]$ is irreducible over $\F_\ell$ but absolutely reducible: we have that $G'\otimes_{\F_\ell} \F_\ell[\sqrt{\epsilon}]$ is diagonalizable (where $\epsilon$ is from the definition of $C_{ns}(\ell)$ in $\S \ref{Sect_SubgpsOfGL2}$), since the eigenvalues of matrices in $C_{ns}(\ell)$ come in pairs $a\pm b\sqrt{\epsilon}$; thus, $\psi$ takes values in $\F_\ell[\sqrt{\epsilon}]^\times$.
\end{remark}
The following short lemma concerns the degree of a field of definition of a torsion point, over the field of definition of its cyclic subgroup. It will be used in the proof of Theorem \ref{Thm_IsogeniesUnderBaseChange}.
\begin{lemma}\label{Lemma_DegreeOfTorsionOverIsogeny}
Let $F/\Q$ be an algebraic extension, $E_{/F}$ an elliptic curve and $P\in E(\oQ)$ a point of order $n\in\Z^+$. Setting $C:=\langle P\rangle$, one has
\[
[F(P):F(C)]\mid \varphi(n),
\]
where $\varphi\colon \Z^+\rightarrow \Z^+$ is Euler's totient function.
\end{lemma}
\begin{proof}
Over the extension $F(C)$, the subgroup $C$ is Galois stable. This induces an isogeny character
\[
r\colon G_{F(C)}\rightarrow (\Z/n\Z)^\times.
\]
One checks that the fixed field of this character is $F(P)$. Thus, modding out by the kernel gives a faithful representation
\[
r\colon \Gal(F(P)/F(C))\hookrightarrow (\Z/n\Z)^\times,
\]
from which we conclude that $[F(P):F(C)]\mid \#(\Z/n\Z)^\times=\varphi(n)$.
\end{proof}

\section{The Proofs of Theorem \ref{Thm_IsogeniesUnderBaseChange} and Proposition \ref{Prop_EvenDegreeFieldOfDefn}}
\subsection{The proof of Theorem \ref{Thm_IsogeniesUnderBaseChange}}
First, we will prove Theorem \ref{Thm_IsogeniesUnderBaseChange} using Proposition \ref{Prop_EvenDegreeFieldOfDefn}.

\begin{proof}[Proof of Theorem \ref{Thm_IsogeniesUnderBaseChange}]
Given the constant $c:=c(F)$ from Proposition \ref{Prop_EvenDegreeFieldOfDefn}, let $B:=B(F)\in\Z^+$ be the product of the primes $p\in\Z^+$ with $p\leq c$.

Let $E_{/F}$ be an elliptic curve. Fix an extension $L/F$ such that $[L:F]$ is coprime to $B$. Let $C\subseteq E(\oQ)$ be a finite cyclic subgroup that is $L$-rational; we will show that $C$ must be $F$-rational. Write out the prime factorization $|C|=\prod_{i=1}^r\ell_i^{a_i}$; then we have a primary decomposition $C=\langle C_1,C_2,\ldots, C_r\rangle$, where $C_i$ is the unique subgroup of $C$ whose order is $\ell_i^{a_i}$. We see that if each $C_i$ is $G_F$-stable, then so is $C$.
Thus, without loss of generality we will assume that $|C|=\ell^n$ is a prime power. 

We will proceed by induction on $n\geq 1$. If $n=1$, then $C$ is an $L$-rational $\ell$-isogeny, and the degree $[F(C):F]\leq \ell+1$. If $\ell\leq c$, then each prime $p$ dividing $[F(C):F]$ is at most $c$, which by definition implies $p\mid B$. However, since $F(C)\subseteq L$ and $[L:F]$ is coprime to $B$, so is $[F(C):F]$; this forces $[F(C):F]=1$. If $\ell> c$, then by Proposition \ref{Prop_EvenDegreeFieldOfDefn} we have that $2\mid [F(C):F]$, and thus $2\mid [L:F]$, which is impossible since $B$ is even and $[L:F]$ is coprime to $B$. We conclude that $C$ is $F$-rational when $n=1$. (We have also just shown that the prime degrees of $L$-rational isogenies are uniformly bounded by $c(F)$.)

Assume next that $n\geq 2$ and the result is true for cyclic $\ell^k$-isogenies, where $1\leq k<n$. This inductive hypothesis implies that $C_{\ell^{n-1}}=\ell C$ is $F$-rational, so that $[F(C):F]=[F(C):F(\ell C)]$. Thus, to show that $C$ is $F$-rational, it suffices to show that if $[F(C):F(\ell C)]\neq 1$, then $[F(C):F(\ell C)]$ is divisible by a prime $p\leq c$, an impossibility. Let us fix a generator $P$ of $C$. Then we have the following diagram of field extensions:
\[
\begin{tikzcd}[every arrow/.append style={dash}]
&F(P)\\
F(\ell P)\arrow[ur]\arrow[ddr]&\\
&F(C)\arrow[uu]\\
&F(\ell C)\arrow[u]
\end{tikzcd}
\]
It follows that 
\[
[F(C):F(\ell C)]=\frac{[F(P):F(\ell P)]\cdot [F(\ell P):F(\ell C)]}{[F(P):F(C)]}.
\]
By Lemma \ref{Lemma_DegreeOfTorsionOverIsogeny} we have $[F(\ell P):F(\ell C)]\mid \varphi(\ell^{n-1})=\ell^{n-2}(\ell-1)$, and by \cite[Proposition 4.6]{GJN20} we know that
$[F(P):F(\ell P)]\mid \ell^2(\ell-1)$, whence it follows that
\[
[F(C):F(\ell C)]\mid \ell^{n}(\ell-1)^2.
\]

Suppose for contradiction that $F(C)\neq F$. Thus, there exists a common prime divisor $p$ of $\ell^{n}(\ell-1)^2$ and $[L:F]$, i.e., there exists prime $p\mid \ell(\ell-1)$ with $p\mid [L:F]$. If $\ell\leq c$ then $p\leq c$, which by definition implies $p\mid B$ -- this is impossible since $[L:F]$ is coprime to $B$. We then must have $\ell>c$, which by Proposition \ref{Prop_EvenDegreeFieldOfDefn} implies that $[F(\ell^{n-1}C):F]$ is even; however, this is impossible since $\ell^{n-1}C$ is $F$-rational by the inductive hypothesis. 
We thus conclude that $C$ is $F$-rational. This completes the proof of Theorem \ref{Thm_IsogeniesUnderBaseChange} using Proposition \ref{Prop_EvenDegreeFieldOfDefn}.
\end{proof}
\subsection{The proof of Proposition \ref{Prop_EvenDegreeFieldOfDefn}}
To finish the proof of Theorem \ref{Thm_IsogeniesUnderBaseChange}, we must now prove Proposition \ref{Prop_EvenDegreeFieldOfDefn}. 

\begin{proof}[Proof of Proposition \ref{Prop_EvenDegreeFieldOfDefn}]

Fix a prime $\ell\in\Z^+$ such that $\ell\geq \max\lbrace 74, 15[F:\Q]+2\rbrace$, $\ell$ is unramified in $F$ and $\ell\not\in S_{F}$.
We must show that for an elliptic curve $E_{/F}$, for any cyclic subgroup $C\in\Cyc(E,\ell)$ one has that $[F(C):F]$ is even.
By Theorem \ref{Thm_BaseChangeImprovement}, our assumptions on $\ell$ imply that, up to conjugacy, $G:=\rho_{E,\ell}(G_{F})$ either equals one of $\GL_2(\ell)$, $N_{ns}(\ell)$ or $G(\ell)$, or that $G$ is contained in $N_s(\ell)$. The groups $\GL_2(\ell)$ and $N_{ns}(\ell)$ act transitively on $E[\ell]\ssm \lbrace O\rbrace$, and thus act transitively on Cyc$(E,\ell)$, which implies that $[F(C):F]=\ell+1$ for all $C\in$ Cyc$(E,\ell)$ when $G\in \lbrace \GL_2(\ell), N_{ns}(\ell)\rbrace$. If $G=G(\ell)$, then $[N_{ns}(\ell):G]=3$, so by Equation \eqref{Eqn_SubgroupOrbitDivisibility} we have $\ell+1\mid 3\cdot [F(C):F]$ for all $C\in \Cyc(E,\ell)$, which implies that each $[F(C):F]$ is even.

Let us assume then that $G\subseteq N_s(\ell)$ with respect to a basis $\lbrace P,Q\rbrace$. By Theorem \ref{Thm_BaseChangeImprovement}, we find that $C_s(\ell)^e\subseteq G$ for some $e\in \lbrace 1,2,3,4,6\rbrace$, and that $G$ contains $Z(\ell)$, the subgroup of scalars of $\GL_2(\ell)$. Note that $G\not\subseteq C_s(\ell)$ since otherwise $E$ has (two distinct) $F$-rational $\ell$-isogenies, which contradicts $\ell\not\in S_{F}$.

Observe that the subgroups $\langle P\rangle$ and $\langle Q\rangle$ have quadratic fields of definition over $F$, since the subgroup $H:=G\cap C_s(\ell)$ has index two in $G$; consequently, both $[F(\langle P\rangle):F]$ and $[F(\langle Q\rangle):F]$ are even. Thus, we will restrict our attention to showing that the remaining order $\ell$ subgroups $C_k:=\langle kP+Q\rangle$, where $0<k<\ell$, have even degree fields of definition over $F$.

Since $G\not\subseteq C_s(\ell)$, there exists an antidiagonal matrix $\gamma:=\begin{bmatrix}
0&b\\
c&0
\end{bmatrix}\in G$. The action of $\gamma$ on $\Cyc(E,\ell)$ sends $C_k$ to $C_{bc^{-1}k^{-1}}$; thus we have $C_{bc^{-1}k^{-1}}\in \oo_G(C_k)$. One can check that if $\gamma\cdot C_k=\gamma \cdot C_m$, then $C_{bc^{-1}k^{-1}}=C_{bc^{-1}m^{-1}}$, which by Lemma \ref{Lemma_DistinctCkSubgroups} implies that $k=m$. Thus, if $C_k\neq  C_{bc^{-1}k^{-1}}$ for all $0<k<\ell$, then all orbits $\oo_G(C_k)$ have even size, as each pair of distinct subgroups $C_k$ and $\gamma\cdot C_k$ lie in the same $G$-orbit; this would then imply that the field of definition of each $C_k$ has even degree over $F$, which finishes the proof. Suppose then that for some $k$ we have $C_k=C_{k^{-1}bc^{-1}}$; by Lemma \ref{Lemma_DistinctCkSubgroups}, this forces $c\equiv bk^{-2}\pmod \ell$, whence we have $\gamma=\begin{bmatrix}
0&b\\
bk^{-2}&0
\end{bmatrix}$. Since $Z(\ell)\subseteq G$, we can scale $\gamma$ by $b^{-1}k$ and deduce that $\begin{bmatrix}
0&k\\
k^{-1}&0
\end{bmatrix}\in G$.

To this end, we may assume that $\gamma:=\begin{bmatrix}
0&b\\
b^{-1}&0
\end{bmatrix}\in G$ for some $1\leq b<\ell$. We check that $\gamma$ sends any subgroup $C_k$ to $C_{b^2k^{-1}}$. By Lemma \ref{Lemma_DistinctCkSubgroups}, we know that $C_k=C_{b^2k^{-1}}$ if and only if $k\equiv \pm b\pmod \ell$. Therefore, to show that all orbits of $\Cyc(E,\ell)$ have  even size, it is enough to show that $C_b$ and $C_{-b}:=C_{\ell-b}$ are in the same $G$-orbit (note that $C_{-b}$ is distinct from $C_b$). 

With some calculation, one can show that $C_b$ and $C_{-b}$ are in the same $G$-orbit if and only if $\begin{bmatrix}
1&0\\
0&-1
\end{bmatrix}\in G$ or $\begin{bmatrix}
0&b\\
-b^{-1}&0
\end{bmatrix}\in G$. Since $\begin{bmatrix}
0&b\\
-b^{-1}&0
\end{bmatrix}=\begin{bmatrix}
1&0\\
0&-1
\end{bmatrix}\cdot \gamma$, we deduce that $C_b$ and $C_{-b}$ are in the same $G$-orbit if and only if $\gamma_0:=\begin{bmatrix}
1&0\\
0&-1
\end{bmatrix}\in G$.

We will show that $\gamma_0\in G$. To do this, we will make a few more reductions.  Let us set $H:=G\cap C_s(\ell)$. 
By Theorem \ref{Thm_BaseChangeImprovement}, we have $C_s(\ell)^e\subseteq H$ for some $e\in \lbrace 1,2,3,4,6\rbrace$. It follows that $\gamma_0\in H$ if $-1$ is an $e$'th power modulo $\ell$. This latter part is true if and only if $\ell=2$ or $\frac{\ell-1}{\gcd(e,\ell-1)}$ is even. Thus, $-1$ is an $e$'th power modulo $\ell$ unless $e\in\lbrace 2,4,6\rbrace$ and $\ell\equiv 3\pmod 4,$ or $e=4$ and $\ell\equiv 5\pmod 8$. 

We will show that in either of these two cases, we still have $\gamma_0\in H$. Let $\det(H)\subseteq\F_\ell^\times$ denote the subgroup of determinants of matrices in $H$, and let $\mu_4(
\F_\ell)$ denote the subgroup of the fourth roots of unity in $\F_\ell$. First, we will show that if $\mu_4(
\F_\ell)\subseteq \det(H)$ then $\gamma_0\in H$.

Assume first that $\ell\equiv 3\pmod 4$.
Observe that $-1\in \det(H)$ if and only if there exists $m=\begin{bmatrix}
a&0\\
0&-a^{-1}
\end{bmatrix}\in H$ for some $a\in \F_\ell^\times$. When this happens, it follows that $am=\begin{bmatrix}
a^2&0\\
0&-1
\end{bmatrix}\in H$, noting that $H$ contains $Z(\ell)$. Since $\ell\equiv 3\pmod 4$, we find that $a^2$ has odd order. Writing this order as $|a^2|$, it follows that $(am)^{|a^2|}=\gamma_0\in H$. Next, assume that $e=4$ and $\ell\equiv 5\pmod 8$; let us write $i\in \F_\ell$ for a primitive fourth root of unity. Then we have $i\in \det(H)$ iff there exists $m=\begin{bmatrix}
a&0\\
0&ia^{-1}
\end{bmatrix}\in H$, and in such a case we have $am=\begin{bmatrix}
a^2&0\\
0&i
\end{bmatrix}\in H$. Thus $(am)^2=\begin{bmatrix}
a^4&0\\
0&-1
\end{bmatrix}\in H$. Since $8\nmid \ell-1$ we have $8\nmid |a|$, and thus $|a^4|$ is odd, so that $(am)^{2|a^4|}=\gamma_0\in H$. We conclude that in any case, if $\mu_4(\F_\ell)\subseteq \det(H)$ then $\gamma_0\in G$.

We will show that $\det(H)=\F_\ell^\times$.
Fix a prime $\Ell\mid \ell$ in $F$ and let $K:=(F)_{\Ell}$ denote the $\Ell$-adic completion of $F$ at $\Ell$. Let us use $I_K$ to denote the inertia group of $K$ (throughout this, we are working with a fixed algebraic closure $\overline{\Q_\ell}$ of $\Q_\ell$). Since $\det\rho_{E,\ell}(I_K)=\chi_\ell(I_K)$ and $\ell$ is unramified in $F$, it follows that $\det\rho_{E,\ell}(I_K)$ surjects onto $\F_\ell^\times$. 

We will argue that the same holds once we replace $F$ by the fixed field of $H$, which we denote by $M_H$.
Fix a prime $\cQ\mid \Ell$ in $M_H$, and let $L:=(M_H)_{\cQ}$ denote the completion of $M_H$ at $\cQ$. Then the inertia group $I_L$ is contained in $H$. In particular, it follows that $\ell$ is unramified in $M_H$, see also \cite[Lemme 2]{Ser72}. Consequently, $L/K$ is unramified, and so $I_K\subseteq I_L$. Thus, from $\det \rho_{E,\ell}(I_K)=\F_\ell^\times$ we deduce that $\det \rho_{E,\ell}(I_L)=\F_\ell^\times$, whence we conclude that $\det(H)=\F_\ell^\times$.

From the work above, we conclude that $\gamma_0\in G$. In particular, $C_b$ and $C_{-b}$ are distinct and in the same $G$-orbit, and so all $G$-orbits of $\Cyc(E,\ell)$ have even size, which is equivalent to the degree $[F(C):F]$ being even for all $C\in \Cyc(E,\ell)$. This concludes our proof.
\end{proof}
\section{The Proof of Theorem \ref{Thm_ModEllRep_ofEllCurve_With1_EllIsogeny}}\label{Sect_FieldsOfDefForBorelUncond}
In this section we will prove Theorem \ref{Thm_ModEllRep_ofEllCurve_With1_EllIsogeny}, which allows us to describe the mod-$\ell$ Galois representation of a non-CM elliptic curve over a number field $F$ with an $F$-rational $\ell$-isogeny, when $\ell$ is uniformly large. Unlike Theorems \ref{Thm_IsogeniesUnderBaseChange} and \ref{Thm_BaseChangeImprovement}, these results do not exclude the possibility of $F$ having rational CM. A similar analysis of isogeny characters can be found in \cite{Gen22}.

We first prove part a. of Theorem \ref{Thm_ModEllRep_ofEllCurve_With1_EllIsogeny}, which provides a relative uniform bound on the index of the mod-$\ell$ Galois representation. 
Since $\rho_{E,\ell}(G_{F})$ is non-diagonalizable, $E$ must be non-CM. 
Let $\langle P\rangle$ be an $F$-rational $\ell$-isogeny on $E$. Then we can choose $Q\in E[\ell]$ such that $\lbrace P,Q\rbrace$ is a basis for $E[\ell]$, and with respect to this basis we have
\[
G:=\rho_{E,\ell,P,Q}(G_{F})=\im \begin{bmatrix}
r&*\\
0&\chi_\ell r^{-1}
\end{bmatrix},
\]
where $\chi_\ell\colon G_{F}\rightarrow \F_\ell^\times$ is the mod-$\ell$ cyclotomic character, and $r$ is the isogeny character of $\langle P\rangle$. 
Let $\Gss$ be the semisimplification of $G$. 
Explicitly, we have
\[
\Gss=\im  \begin{bmatrix}
r&0\\
0&\chi_\ell r^{-1}
\end{bmatrix}.
\]
Since $E$ has exactly one $F$-rational $\ell$-isogeny, we know that $G$ is not diagonalizable, thus by \cite[Lemma 4]{Gen22} we have $\Gss\subseteq G$. It also follows by the proof of \cite[Lemma 4]{Gen22} that $\gamma:=\begin{bmatrix}
1&1\\
0&1
\end{bmatrix}\in G$, and thus $G=\langle \gamma, \Gss\rangle$. Since we also have $B(\ell)=\langle \gamma, C_s(\ell)\rangle$, we deduce that 
\begin{equation}\label{Eqn_IndexEqualsSemiSimpliIndex}
[B(\ell):G]=[C_s(\ell):\Gss].
\end{equation}
The rest of this proof will be showing that $[C_s(\ell):\Gss]\mid 864[F:\Q]$.

Since we are assuming that GRH is true, by Theorem \ref{Thm_LV} the existence of the $F$-rational $\ell$-isogeny $\langle P\rangle$ implies the following: there exists a CM elliptic curve $E'_{/F}$ whose CM field $K$ is contained in $F$, such that $E'$ has an associated character $\psi\colon G_{F}\rightarrow \overline{\F_\ell}^\times$ which satisfies
\begin{equation}\label{Eqn_12PowersIsogCharLV}
r^{12}=\psi^{12}.
\end{equation}
By our initial hypotheses, $\ell$ is unramified in the ring $\oo:=\End(E')$, regarded as an order of the CM field $K$. By \cite[Corollary 1.5]{BC20}, up to conjugation we have the bound
\begin{equation}\label{Eqn_UniformBoundOnModEllIndexCM}
[C_\ell(\oo):G']\mid 6[F:\Q]
\end{equation}
where $G':=\rho_{E',\ell}(G_{F})$ and $C_\ell(\oo)$ is either $C_s(\ell)$ or $C_{ns}(\ell)$ depending on whether $\ell$ is split or inert in $\oo$, respectively.

We will show that $\ell$ must be split in $\oo$. Assume to the contrary that $\ell$ is inert in $\oo$; then as noted in Remark \ref{Rmk_AssocChar}, the elements of $G'$ have eigenvalues over $\F_{\ell^2}:=F_\ell[\sqrt{\epsilon}]$ which are Galois conjugates of one another. Writing 
\[
G'\otimes_{\F_\ell}\F_{\ell^2}=\im \begin{bmatrix}
\psi&0\\
0&\psi'
\end{bmatrix},
\]
it follows that 
\begin{equation}\label{Eqn_SizeOfCMGalRepIsCharInert}
\# G'=\# \psi(G_{F}).
\end{equation}

Now, for any nonzero torsion point $R'\in E'[\ell]$, since $G'\subseteq C_{ns}(\ell)$ we have by Equations \eqref{Eqn_SubgroupOrbitDivisibility} and \eqref{Eqn_UniformBoundOnModEllIndexCM} the orbit size divisibility 
\[
\# \oo_{C_{ns}(\ell)}(R')\mid 6[F:\Q]\cdot [F(R'):F],
\]
noting that $\# \oo_{G'}(R)=[F(R'):F]$.
Since the action of $C_{ns}(\ell)$ on $E'[\ell]\ssm \lbrace O\rbrace$ is transitive, this implies that
\[
\ell^2-1\mid 6[F:\Q]\cdot [F(R'):F].
\]
However, from $r^{12}=\psi^{12}$ we find that 
\[
\gcd(\# \psi(G_{F}),12)\cdot \# r(G_{F})=\gcd(\# r(G_{F}), 12)\cdot \# \psi(G_{F}),
\]
and thus 
\[
\#\psi(G_{F})\mid 12\cdot \# r(G_{F})\mid 12(\ell-1).
\]
Since $[F(R'):F]\mid\# G'$, and because $\# G'=\#\psi(G_{F})$ by Equation \eqref{Eqn_SizeOfCMGalRepIsCharInert}, we deduce that
\[
\ell^2-1\mid 6[F:\Q]\cdot 12(\ell-1),
\]
and so
\[
\ell+1\mid 72[F:\Q].
\]
However, this is impossible since $\ell\geq 72[F:\Q]$. We deduce that $\ell$ splits in $\oo$.

Since $\ell$ splits in $\oo$, we can choose a basis for $E'[\ell]$ such that $G'$ has the form
\[
G'=\im  \begin{bmatrix}
\psi&0\\
0&\chi_\ell\psi^{-1}
\end{bmatrix}.
\]
Thus, from Equation \eqref{Eqn_12PowersIsogCharLV} it follows that $(\Gss)^{12}=(G')^{12}$.
With this, we have the containments $C_s(\ell)\supseteq \Gss\supseteq (\Gss)^{12}=(G')^{12}$, whence it follows that
\[
[C_s(\ell):\Gss]\mid [C_s(\ell): (G')^{12}].
\]
Since $[G':(G')^{12}]\mid 144$, this shows that
\[
[C_s(\ell):\Gss]\mid 144 [C_s(\ell):G'].
\]
By Equation \eqref{Eqn_UniformBoundOnModEllIndexCM},
we thus conclude that
\[
[C_s(\ell):\Gss]\mid 864[F:\Q].
\]
Since $[B(\ell):G]=[C_s(\ell):\Gss]$ by Equation \eqref{Eqn_IndexEqualsSemiSimpliIndex}, this concludes our proof of part a.

Next, we will prove part b. of Theorem \ref{Thm_ModEllRep_ofEllCurve_With1_EllIsogeny}.
With our assumptions, we can choose a basis $\lbrace P,Q\rbrace$ of $E[\ell]$ such that $G:=\rho_{E,\ell,P,Q}(G_{F})\subseteq B(\ell)$. Since $E$ has exactly one $F$-rational $\ell$-isogeny, we know that $G$ is not diagonalizable.  Thus, Equation \eqref{Eqn_SubgroupOrbitDivisibility} implies that for each subgroup $C\in \Cyc(E,\ell)$, one has the orbit size divisibility
\[
\#\oo_{B(\ell)}(C)\mid [B(\ell):G]\cdot [F(C):F].
\]
By part a., this implies that
\[
\#\oo_{B(\ell)}(C)\mid 864[F:\Q]\cdot [F(C):F].
\]

The orbits of $\Cyc(E,\ell)$ under $B(\ell)$ are easy to compute. We check that $\oo_{B(\ell)}(\langle P\rangle)=\lbrace \langle P\rangle \rbrace$ and $\oo_{B(\ell)}(\langle P+Q\rangle)=\lbrace \langle kP+Q\rangle :0\leq k<\ell\rbrace$, with the latter orbit having size $\ell$. Since $\ell>\max\lbrace 2,3,[F:\Q]\rbrace$ (from $\ell\geq c(F)$), we have $\ell\nmid 864[F:\Q]$; thus, for $C\neq \langle P\rangle$ we have $\ell\mid [F(C):F]$, which by $[F(C):F]\leq \ell+1$ forces $[F(C):F]=\ell$.

\end{document}